\documentclass[12pt, reqno, a4paper]{amsart}

\usepackage{ amssymb, amsmath, enumerate, amsfonts, amsthm, mathrsfs, url, bm, mathtools}

\usepackage{xcolor}  	
\usepackage{hyperref}
\hypersetup{
colorlinks,
   linkcolor={cyan!80!black},
   citecolor={cyan!80!black},
 urlcolor={cyan!80!black}
}

\usepackage{color}

\usepackage[margin=1in]{geometry}

\RequirePackage{doi}

\usepackage{amscd}
\usepackage{amsfonts}
\usepackage{float}
\usepackage{color}
\usepackage[
backend=biber,
style=alphabetic,
]{biblatex}
\usepackage{bookmark}

\renewbibmacro{in:}{}
\DeclareFieldFormat{title}{#1}

\DeclareFieldFormat[article]{title}{\mkbibemph{#1}}       % Articles
\DeclareFieldFormat[incollection]{title}{\mkbibemph{#1}}  % Incollection chapters
\DeclareFieldFormat[book]{title}{\mkbibemph{#1}}          % Standalone books
\DeclareFieldFormat[incollection]{booktitle}{#1}          % Plain for incollection book titles
\DeclareFieldFormat[article]{journaltitle}{#1}            % Plain for journals

\AtEveryBibitem{%
  \ifentrytype{misc}{\DeclareFieldFormat{title}{\mkbibemph{#1}}}{}}

\DeclareFieldFormat{eprint:eprint}{arXiv:\href{https://arxiv.org/abs/#1}{#1}}

\DeclareFieldFormat[inproceedings]{title}{\mkbibemph{#1}}

\DeclareFieldFormat[inproceedings]{booktitle}{#1}

\usepackage{amssymb}
\newtheorem{theorem}{Theorem}[section]
\newtheorem{lemma}{Lemma}[section]
\newtheorem{example}{Example}[section]
\newtheorem{corollary}{Corollary}[section]
\newtheorem{proposition}{Proposition}[section]

\theoremstyle{definition}
\newtheorem{definition}{Definition}[section]

\theoremstyle{remark}

\numberwithin{equation}{section}

\newcommand{\C}{\mathbb C}
\newcommand{\R}{\mathbb R}
\newcommand{\Q}{\mathbb Q}
\newcommand{\Z}{\mathbb Z}
\newcommand{\N}{\mathbb N}
\newcommand{\E}{\mathbb E}
\newcommand{\T}{\mathbb T}
\newcommand{\M}{\mathcal M}

\newcommand{\Mod}[1]{\ (\mathrm{mod}\ #1)}

\renewcommand{\Re}{\operatorname{Re}}

\renewcommand{\leq}{\leqslant}
\renewcommand{\geq}{\geqslant}

\begin{document}

\title[Multiplicative recurrence of Piatetski--Shapiro and nil-Bohr sets]{Multiplicative recurrence of Piatetski--Shapiro and nil-Bohr sets}

%    Only \author and \address are required; other information is
%    optional.  Remove any unused author tags.

%    author one information
% \author[short version for running head]{name for top of paper}
\author{Sun-Kai Leung}
\address{Mathematical Institute, University of Oxford, Andrew Wiles Building\\ Radcliffe Observatory Quarter, Woodstock Rd\\
Oxford OX2 6GG\\
United Kingdom}
\email{sunkaileung@gmail.com}

%    author two information

  %  \subjclass is required.

\date{}

\dedicatory{}

%    Abstract is required.
\begin{abstract}
Adapting the circle method developed by Frantzikinakis--Klurman--Moreira, we establish a sufficient criterion for a set of natural numbers to be multiplicatively recurrent. We apply this criterion to Piatetski--Shapiro sets with small exponents and to regular irrational nil-Bohr sets. We also study the value distribution of unimodular completely multiplicative functions along these sets.
\end{abstract}

\maketitle

\section{Introduction}

\textit{Multiplicative Ramsey theory} studies recurrence and partition regularity
for the multiplicative structure of the natural numbers. Following
Bergelson \cite{MR2191223}, this subject has been developed using ergodic
theory and the structure of multiplicative functions, notably by
Frantzikinakis--Host \cite{MR3556289}. For instance, Donoso, Le, Moreira,
and Sun \cite{MR4594405} proved that an arithmetic progression
$a\N+b$ is a \textit{set of multiplicative recurrence} if and only if
$a\mid b(b-1)$. Further examples arise from linear fractional sequences \cite{MR4960398,leung2025multiplicativerecurrencemobiustransformations}
and homogeneous quadratic forms
\cite{frantzikinakis2023partitionregularitypythagoreanpairs,
frantzikinakis2024partitionregularitygeneralizedpythagorean}.
We refer to \cite{frantzikinakis2024partitionregularityhomogeneousquadratics}
for an exposition of the latter developments.

We first recall the notion of recurrence used throughout the paper.

\begin{definition}[Set of multiplicative recurrence]
A set $R\subseteq\Q_{>0}$ is called a
\textit{set of multiplicative recurrence} or \textit{multiplicatively recurrent} if, for every
measure-preserving action $T=(T_n)_{n\in\N}$ of $(\N,\times)$ on a
probability space $(X,\mathcal B,\mu)$ and every measurable set
$B\in\mathcal B$ of positive measure, there exist distinct
$m,n \in \N$ such that $ m/n\in R$ and
\begin{align*}
 \mu(T_m^{-1}B\cap T_n^{-1}B)>0.
\end{align*}
Here $T_1=\rm id$ and $T_{mn}=T_m\circ T_n$ for all $m,n \in \N$. 
\end{definition}

Our sufficient criterion for multiplicative recurrence is formulated in terms of
means of unimodular completely multiplicative functions. Write
\begin{align*}
 \M:=\{f:\N\to\T:f(mn)=f(m)f(n)\text{ for all }m,n\in\N\},
\end{align*}
where $\T:=\{z\in\C:|z|=1\}$. Following
\cite[equation (2.11)]{frantzikinakis2023partitionregularitypythagoreanpairs},
we consider the \textit{multiplicative F\o{}lner sequence}
\begin{align}\label{eq:folner}
 \Phi_K:=\left\{\prod_{p\leq K}p^{v_p}:K< v_p \leq 2K\right\}
 \qquad(K\in\N).
\end{align}
In particular, for
every fixed $r \in \Q_{>0}$, we have
\begin{align*}
\lim_{K \to \infty} \frac{|r\Phi_K\mathbin{\cap}\Phi_K|}{|\Phi_K|} =1.
\end{align*}

In this paper, we introduce the following notion of multiplicative pseudorandomness, comparing means of multiplicative functions
on a set and its dilates in order to apply the circle method developed by
Frantzikinakis--Klurman--Moreira
\cite{frantzikinakis2023partitionregularitypythagoreanpairs}.

\begin{definition}[Multiplicative pseudorandom set]\label{def:mpr}
A set $R\subseteq\N$ is called a \textit{multiplicative pseudorandom set} or \textit{multiplicatively pseudorandom}
if, for every sufficiently large $K \in \N$ and every $Q\in\Phi_K$,
the set $R\cap Q\N$ is infinite and, for every $f\in\M$, we have
\begin{align*}
 \lim_{N\to\infty}\left|
 \E_{\substack{n\leq N\\Qn\in R}}f(n)
 -\E_{\substack{n\leq N\\n\in R}}f(n)\right|=0.
\end{align*}
Each expectation is normalized by the cardinality of the set in its
subscript.
\end{definition}

%The limit in \eqref{eq:mpr} concerns the difference of two averages.
%Neither average is required to converge, and no uniformity in the
%function is assumed. The examples below satisfy the stronger version
%with the supremum over all $f\in\M$ inside the limit.

Our first result gives a sufficient criterion for multiplicative recurrence.

\begin{theorem}\label{thm:main}
Every multiplicatively pseudorandom set is multiplicatively recurrent.
\end{theorem}

The multiplicative Furstenberg correspondence principle
\cite[Theorem~2.1]{MR4594405} gives the following
consequence for finite colorings.

\begin{corollary}\label{cor:colour}
Let $R\subseteq\N$ be a multiplicative pseudorandom set. Then for every
finite partition $\N=C_1\sqcup\cdots\sqcup \, C_s$, the intersection
$(C_i/C_i)\cap(R\setminus\{1\})$ is non-empty for some integer
$1\leq i\leq s$. In other words,
the set $R$ is a set of \textit{topological multiplicative recurrence}.
\end{corollary}

For $f_1,\ldots,f_r\in\M,$ applying the criterion to the rotation
\begin{align*}
T_n(z_1,\ldots,z_r):=(f_1(n)z_1,\ldots,f_r(n)z_r)
\end{align*}
on $\T^r$ with Haar probability measure and a small neighbourhood of the identity gives the following simultaneous approximation.

\begin{corollary}\label{cor:heil}
Let $R\subseteq\N$ be a multiplicative pseudorandom set. For every
$r \in \N$ and every $f_1,\ldots,f_r\in\M$,
we have
\begin{align*}
 \liminf_{\substack{n\to\infty\\n\in R}}
 \max_{1\leq j\leq r}|f_j(n)-1|=0.
\end{align*}
\end{corollary}

The following Khintchine-type refinement strengthens
Theorem~\ref{thm:main} and gives positive relative lower density under
two natural counting hypotheses. For an infinite set $R\subseteq\N$, write
\begin{align*}
 R(x)&:=|R\cap[1,x]|,\\
 \underline d_R(S)&:=\liminf_{x\to\infty}
 \frac{|S\cap R\cap[1,x]|}{R(x)}.
\end{align*}

\begin{theorem}\label{thm:suff}
Let $R\subseteq\N$ be a multiplicative pseudorandom set, and let
$T=(T_n)_{n\in\N}$ be a measure-preserving multiplicative action on a
probability space $(X,\mathcal B,\mu)$. For every measurable set
$B\in\mathcal B$ of positive measure and every $\varepsilon>0$, the set
\begin{align*}
 \mathcal T_R(\mu,T,B;\varepsilon)
 :=\{n\in R:\mu(T_n^{-1}B\cap B)>\mu(B)^2-\varepsilon\}
\end{align*}
is infinite. Suppose in addition that, for every $Q\in\N$ and $C\geq1$, we have
\begin{gather}
 |R\cap Q\N\cap[1,x]|\gg_{R,Q}R(x),\label{eq:nolocal}\\
 R(x/C)\gg_{R,C}R(x)\label{eq:regular}
\end{gather}
as $x\to\infty$. Then
$\underline d_R(\mathcal T_R(\mu,T,B;\varepsilon))>0$.
\end{theorem}

The condition \eqref{eq:nolocal} requires a positive proportion of the
set to lie in each fixed divisibility class. The condition
\eqref{eq:regular} prevents its counting function from increasing too
rapidly under a fixed dilation. 

In the additive setting, Khintchine's recurrence theorem
\cite{MR1556883} gives a \textit{syndetic} set of return times at the threshold
$\mu(B)^2-\varepsilon$.\footnote{A subset of $\N$ is \textit{syndetic} if it has bounded gaps.}
The conclusion of Theorem~\ref{thm:suff} is instead expressed by relative
lower density. Nor can it be replaced by syndeticity after parametrizing
a Piatetski--Shapiro set by its defining index, even though it is
multiplicatively recurrent for $1<c<6/5$ by Theorem~\ref{thm:ps}.
Indeed, for $t\ne0$, the
multiplicative rotation $T_mz=m^{it}z$ on $\T$ and a sufficiently short
arc $B$ give arbitrarily long intervals of integers $n$ for which
$T_{\lfloor n^c\rfloor}^{-1}B\cap B$ is empty.

We apply this criterion to two classes of sets for which the
multiplicative structure is less apparent. The first consists of integer parts of
fractional powers.

\begin{definition}[Piatetski--Shapiro set]
For a non-integral $c>1$, the set
\begin{align*}
 \mathcal P_c:=\{\lfloor n^c\rfloor:n\in\N\}
\end{align*}
is called a \textit{Piatetski--Shapiro set}.
\end{definition}

The second class is defined using nilmanifolds. Throughout the paper,
a \textit{nilmanifold} is a compact quotient $X=G/\Gamma$ of a connected,
simply connected nilpotent Lie group by a discrete cocompact subgroup.
We write $m_X$ for its Haar probability measure. \textit{Polynomial sequences}
$g:\Z\to G$ are understood with respect to a \textit{rational filtration},
as in \cite[Section~5.1]{MR3556289}. Equivalently, for a filtration
$(G_j)_{j=0}^{s+1}$ of degree $s$, they admit the expansion
\begin{align*}
 g(n)=g_0g_1^{\binom n1}\cdots g_s^{\binom ns},
 \qquad g_j\in G_j\qquad(0\leq j\leq s)
\end{align*}
given in \cite[equation~(5.1)]{MR3556289}. See also
\cite[Lemma~6.7]{MR2877065}.
The orbit $(g(n)\Gamma)_{n\in\Z}$ is \textit{totally equidistributed}
if, for every positive integer $a$, every integer $b$ and every
continuous function $F:X\to\C$, we have
\begin{align*}
 \lim_{N\to\infty}\E_{n\leq N}F(g(an+b)\Gamma)
 =\int_XF\,dm_X.
\end{align*}
Nil-Bohr sets were introduced by Host--Kra \cite{MR2755924}.
For their formulation using polynomial nilsequences, see
Shao--Ter\"av\"ainen \cite[Definition~2.2]{shao2021bombieri}.
We use the following regular irrational variant.

\begin{definition}[Regular irrational nil-Bohr set]
Let $X=G/\Gamma$ and $g:\Z\to G$ be as above, and suppose that the
orbit is totally equidistributed. For a Borel set $A\subseteq X$
satisfying $m_X(A)>0$ and $m_X(\partial A)=0$, we call
\begin{align*}
 \mathcal B_{X,g}(A):=\{n\in\N:g(n)\Gamma\in A\}
\end{align*}
a \textit{regular irrational nil-Bohr set}.
\end{definition}

The adjective \textit{regular} refers to the boundary condition, while \textit{irrational}
refers to total equidistribution, in keeping with the terminology of
\cite[Theorem~2.2]{MR3556289}.

For the Piatetski--Shapiro application, we first restrict to smooth
integers in order to obtain a bilinear factorization.

\begin{definition}[$\delta$-smooth numbers]
Given $\delta>0$, we put
\begin{align*}
 \mathcal S_\delta:=\{n\in\N:P^+(n)\leq n^\delta\},
\end{align*}
where $P^+(n)$ is the largest prime factor of $n$ and $P^+(1):=1$.
The elements of this set are called \textit{$\delta$-smooth numbers}.
\end{definition}

Heath--Brown's bilinear estimate \cite{MR698168}, together with a
factorization of smooth integers, gives the first application.

\begin{theorem}\label{thm:ps}
Let $1<c<6/5$. There exists $\delta=\delta(c)>0$ such that the intersection
$\mathcal P_c\cap\mathcal S_\delta$ is multiplicatively pseudorandom.\footnote{Adapting the argument in
\cite[Section~2]{MR3705270}, one can show that without the smoothness
restriction, Theorem~\ref{thm:ps} is false in general.}
In particular, the set $\mathcal P_c$ is multiplicatively
recurrent.
\end{theorem}

For nil-Bohr sets, the required estimate follows from the uniform
orthogonality theorem of Frantzikinakis--Host
\cite[Theorem~2.2]{MR3556289}.

\begin{theorem}\label{thm:nilbohr}
Every regular irrational nil-Bohr set is multiplicatively pseudorandom,
and hence is a set of multiplicative recurrence.
\end{theorem}

A \textit{generalized polynomial} is a function obtained from ordinary real
polynomials by finitely many applications of addition, multiplication,
and the floor function, as in \cite{MR1247389,MR2318563}. The representation theorem of
Bergelson--Leibman \cite[Theorem~A$^*$]{MR2318563} realizes bounded
generalized polynomials as coordinate functions of polynomial orbits
on nilmanifolds. Together with the factorization theorem for polynomial
orbits \cite[Corollary~1.12]{MR2877065}, this gives the following
consequence.

\begin{corollary}\label{cor:genpoly}
Let $p:\Z\to\R$ be a generalized polynomial such that
$(p(an+b))_{n\in \N}$ is equidistributed modulo one for every
$a \in \N$ and every $b \in \Z$. For an interval
$I \subseteq [0,1)$ of positive length, the set
\begin{align*}
 \mathcal{L}_{p}(I):=\{n\in\N: \{p(n)\}\in I\}
\end{align*}
is multiplicatively pseudorandom, and hence is a set of multiplicative
recurrence.
\end{corollary}

We next record three concrete examples of Corollary~\ref{cor:genpoly},
which are verified in Section~\ref{sec:genpoly}.

\begin{example}[Inhomogeneous Beatty set]\label{ex:beatty}
Let $\alpha>1$ be irrational and $\beta\geq0$. The set
\begin{align*}
 \mathcal B_{\alpha,\beta}
 :=\{\lfloor\alpha n+\beta\rfloor:n\in\N\}
\end{align*}
is multiplicatively pseudorandom and recurrent.
\end{example}

The next example is a \textit{bracket quadratic level set} arising from a polynomial
orbit on the Heisenberg nilmanifold.

\begin{example}[Bracket quadratic level set]\label{ex:bracket}
Suppose that $1,\alpha,\beta$ are linearly independent over $\Q$.
For any interval $I\subseteq[0,1)$ of positive length, the set
\begin{align*}
 \mathcal{Q}_{\alpha,\beta}:=\{n\in\N:\{\beta n\lfloor\alpha n\rfloor\}\in I\}
\end{align*}
is multiplicatively pseudorandom and recurrent.
\end{example}

The generalized-polynomial criterion also includes ordinary polynomial phases.

\begin{example}[Polynomial level set]\label{ex:polynomial}
Let $P\in\R[x]$ satisfy $P-P(0)\notin\Q[x]$, and let
$I\subseteq[0,1)$ be an interval of positive length. Then the set
\begin{align*}
 \mathcal{L}_P(I):=\{n\in\N:\{P(n)\}\in I\}
\end{align*}
is multiplicatively pseudorandom and recurrent.
\end{example}

For both classes in Theorems~\ref{thm:ps} and \ref{thm:nilbohr}, the
counting estimates \eqref{eq:pscond} and \eqref{eq:bscond}, together
with the corresponding counting asymptotics, verify \eqref{eq:nolocal}
and \eqref{eq:regular}. Moreover, the estimate \eqref{eq:pscond2}
shows that the smooth subset in Theorem~\ref{thm:ps} has positive
relative density in $\mathcal P_c$. Consequently, the set
$\mathcal T_{\mathcal P_c}(\mu,T,B;\varepsilon)$ has positive lower
density relative to the full Piatetski--Shapiro set as well.

Finally, Section~\ref{sec:dist} determines the limit points of
$(f(n))_{n\in R}$ for every multiplicative pseudorandom set and every
$f\in\M$. They form the full closed subgroup generated by the values of
$f$. We also characterize equidistribution on the nil-Bohr sets and on
the smooth Piatetski--Shapiro sets in terms of \textit{pretentious distance}.
The final section discusses the local restrictions in our criterion,
and the appendix records the relation between the order of a
multiplicative function and ergodicity of its rotation action.

\medskip
\noindent\textit{Notation.}
We use the standard big $O$ and little $o$ notations, as well as the
Vinogradov notations $\ll,\gg$ and the Hardy notation $\asymp$. The implied
constants depend only on the indicated parameters unless otherwise
specified. For $t\in\R$, we write $e(t):=\exp(2\pi it)$ and
$\psi(t):=\{t\}-\frac12$.
For $n\geq2$, we write $P^+(n)$ and $P^-(n)$ for the largest and smallest
prime factors of $n$, respectively.
For a non-empty finite set $S$ and a function $a:S\to\C$, we write
\begin{align*}
 \E_{n\in S}a(n):=\frac1{|S|}\sum_{n\in S}a(n).
\end{align*}
For $R\subseteq\N$ and $Q\in\N$, we write
$R_Q:=\{n\in\N:Qn\in R\}$.

\section{Proof of Theorem~\ref{thm:suff}}

We use harmonic analysis on the multiplicative group of positive rationals,
following Elliott \cite{MR766558,MR3731311} and the approach of
Frantzikinakis--Klurman--Moreira
\cite{frantzikinakis2023partitionregularitypythagoreanpairs,frantzikinakis2024partitionregularitygeneralizedpythagorean}.
Every $f\in\M$ extends uniquely to $\Q_{>0}$ by setting
$f(r/s):=f(r)\overline{f(s)}$ for $r,s\in\N$. The extension satisfies
$f(rs)=f(r)f(s)$ for all $r,s\in\Q_{>0}$ and is therefore a character of
$(\Q_{>0},\times)$.
Thus the space $\M$, equipped with the topology of pointwise convergence,
is the compact Pontryagin dual of the discrete group
$(\Q_{>0},\times)$.

For the F\o{}lner sets in \eqref{eq:folner}, the following orthogonality relation for multiplicative functions is the basic ingredient of the \textit{$Q$-trick} of
Frantzikinakis--Klurman--Moreira
\cite[Section~4]{frantzikinakis2023partitionregularitypythagoreanpairs}.
See also \cite[Section~6]{frantzikinakis2024partitionregularityhomogeneousquadratics}.

\begin{lemma} \label{lem:mult}
For every $f\in\M$, we have
\[
\lim_{K\to\infty}\E_{Q\in\Phi_K}f(Q)
=\begin{cases}
1&\mbox{{\normalfont if $f\equiv1$,}}\\
0&\mbox{{\normalfont otherwise.}}
\end{cases}
\]
\end{lemma}

\begin{proof}
This is \cite[Lemma 3.2]{frantzikinakis2023partitionregularitypythagoreanpairs}.
\end{proof}

We next combine this identity with multiplicative pseudorandomness.
An immediate obstacle is that means of multiplicative functions need not exist,
as the example $f(n)=n^{it}$ with $t\ne0$ already shows on $\N$. Nevertheless,
our definition requires only the difference of the two means to tend
to zero, which is sufficient for the following spectral identity.

\begin{proposition} \label{prop}
Let $R$ be a multiplicatively pseudorandom set and let $\sigma$ be a finite
positive Borel measure on $\M$. Then
\begin{equation*}
\lim_{K\to\infty}\limsup_{N\to\infty}
\left|
\E_{Q\in\Phi_K}
\E_{\substack{n\leq N\\Qn\in R}}
\int_{\M}f(Qn)\,d\sigma(f)-\sigma(\{1\})
\right|=0.
\end{equation*}
\end{proposition}

\begin{proof}
Write
\begin{align*}
A_Q(N;f)&:=\E_{\substack{n\leq N\\Qn\in R}}f(n),\\
A_1(N;f)&:=\E_{\substack{n\leq N\\n\in R}}f(n).
\end{align*}
Fix a sufficiently large $K$. For each $Q\in\Phi_K$ and $f\in\M$,
the definition gives $A_Q(N;f)-A_1(N;f)\to0$ as $N\to\infty$.
The integrands are bounded by two, and since $\Phi_K$ is finite, dominated
convergence gives
\begin{gather*}
\E_{Q\in\Phi_K}
\E_{\substack{n\leq N\\Qn\in R}}
\int_{\M}f(Qn)\,d\sigma(f)\\
=
\int_{\M}\left(\E_{Q\in\Phi_K}f(Q)\right)A_1(N;f)\,d\sigma(f)
+o_{K,\sigma}(1).
\end{gather*}
The contribution of $f\equiv1$ is $\sigma(\{1\})$, while the absolute value
of the remaining integral is at most
\[
\int_{\M\setminus\{1\}}\left|\E_{Q\in\Phi_K}f(Q)\right|\,d\sigma(f).
\]
Lemma~\ref{lem:mult} and dominated convergence show that this expression
tends to zero as $K\to\infty$, proving the proposition.
\end{proof}

We use the standard spectral representation for multiplicative actions.

\begin{lemma} \label{lem:spectral}
Let $(T_n)_{n\in\N}$ be a measure-preserving action of
$(\N,\times)$ on $(X,\mathcal B,\mu)$, and let $B\in\mathcal B$.
There exists a finite positive Borel measure $\sigma$ on $\M$ such that for every $n \in \N,$ we have $\sigma(\{1\})\geq\mu(B)^2$ and
\begin{align*}
\mu(T_n^{-1}B\cap B)=\int_{\M}f(n)\,d\sigma(f).
\end{align*}
\end{lemma}

\begin{proof}
This is the Bochner--Herglotz representation used in
\cite[Section~2.1]{frantzikinakis2023partitionregularitypythagoreanpairs}.
By measure preservation, the map
\begin{align*}
 a/b\longmapsto
 \langle 1_B\circ T_a,1_B\circ T_b\rangle_{L^2(\mu)}
\end{align*}
is well defined and positive definite on $(\Q_{>0},\times)$.
Applying the representation to $1_B-\mu(B)$ gives a finite positive
measure $\nu$ representing the centered correlations. Then
$\sigma:=\nu+\mu(B)^2\delta_1$ has the asserted representation and
$\sigma(\{1\})\geq\mu(B)^2$.
\end{proof}

We now deduce recurrence and its relative-density refinement by fixing
$K$ before choosing an element of $\Phi_K$ for each sufficiently large $N$.

\begin{proof}[Proof of Theorem~\ref{thm:suff}]
Put $a:=\mu(B)^2$. For $m\in\N$, write
\begin{align*}
c(m)&:=\mu(T_m^{-1}B\cap B),\\
S&:=\{m\in R:c(m)>a-\varepsilon\}.
\end{align*}
It suffices to consider $0<\varepsilon<a$. Applying
Lemma~\ref{lem:spectral} and Proposition~\ref{prop}, we may fix a
sufficiently large $K$ such that, for every sufficiently large $N$, we have
\[
\E_{Q\in\Phi_K}
\E_{\substack{n\leq N\\Qn\in R}}c(Qn)>a-\frac\varepsilon2.
\]
For each such $N$, there exists an element $Q=Q_N\in\Phi_K$ such that the
inner average exceeds $a-\varepsilon/2$. Since
$0\leq c(m)\leq1$ and $c(m)\leq a-\varepsilon+1_S(m)$ for $m\in R$, we obtain
\begin{equation} \label{eq:ratio}
|S\cap[1,QN]|
\geq\frac\varepsilon2\,|R_Q\cap[1,N]|.
\end{equation}
Each set $R_Q$ is infinite and $\Phi_K$ is finite. Consequently,
the minimum of the counts on the right-hand side tends to infinity with
$N$, proving that $S$ is infinite.

Suppose now that \eqref{eq:nolocal} and \eqref{eq:regular} hold,
and put $L:=\max\Phi_K$. For a sufficiently large real $x$, apply
\eqref{eq:ratio} with $N:=\lfloor x/L\rfloor$. The selected element
satisfies $QN\leq x$. Since $QN\geq N\geq x/(2L)$, the two counting
hypotheses give
\begin{align*}
|R_Q\cap[1,N]|
&=|R\cap Q\N\cap[1,QN]|\\
&\gg_{R,Q}R(QN)
 \geq R(x/(2L))
 \gg_{R,L}R(x).
\end{align*}
The implied constants can be made uniform over the finite set $\Phi_K$.
It follows that $|S\cap[1,x]|\gg_{R,K}\varepsilon|R\cap[1,x]|$ for all
sufficiently large $x$, giving the required positive relative lower density.
\end{proof}

\section{Proof of Theorem~\ref{thm:ps}}

The proof reduces to comparing multiplicative means on smooth
Piatetski--Shapiro sets with weighted means on smooth numbers. The comparison
is uniform in the multiplicative function and holds for every fixed dilation.

\begin{proposition}\label{prop:ps}
Let $1<c<6/5$ and put $\gamma:=c^{-1}$. There exists
$\delta=\delta(c)>0$ such that, for every fixed $Q\in\N$, we have
\begin{align*}
 \sum_{\substack{n\leq N\\Qn\in\mathcal P_c\cap\mathcal S_\delta}}f(n)
 =\gamma Q^{\gamma-1}
   \sum_{\substack{n\leq N\\n\in\mathcal S_\delta}}
       \frac{f(n)}{n^{1-\gamma}}+o_{Q,c}(N^\gamma)
\end{align*}
as $N\to\infty$, uniformly for $f\in\M$.\footnote{Adapting the
argument in \cite[Section~2]{MR3705270}, one can show that without the
smoothness restriction, Proposition~\ref{prop:ps} is false in general.}
\end{proposition}

We first deduce Theorem~\ref{thm:ps} from the proposition. Let $\rho$ denote
the \textit{Dickman--de Bruijn function}, defined by $\rho(u)=1$ for $0\leq u\leq1$
and
\begin{align*}
 u\rho'(u)+\rho(u-1)=0\qquad(u>1).
\end{align*}
The classical smooth-number asymptotic
\cite[Theorem~14.5]{MR3971232} gives
\begin{align}\label{eq:ps-smooth-density}
 |\{n\leq N:n\in\mathcal S_\delta\}|
       =(\rho(\delta^{-1})+o_\delta(1))N.
\end{align}
Indeed, for $\varepsilon N<n\leq N$, the cutoff $n^\delta$ lies between
$(\varepsilon N)^\delta$ and $N^\delta$. Applying the usual asymptotic with
these two fixed cutoffs and then letting $\varepsilon\to0^+$ proves the
assertion. Partial summation and Proposition~\ref{prop:ps} with $f\equiv1$
therefore imply
\begin{align*}
 |\{n\leq N:Qn\in\mathcal P_c\cap\mathcal S_\delta\}|
       =(\rho(\delta^{-1})+o_{Q,c}(1))Q^{\gamma-1}N^\gamma.
\end{align*}
Consequently, uniformly for $f\in\M$, we have
\begin{align}\label{eq:sumps}
 \E_{\substack{n\leq N\\Qn\in\mathcal P_c\cap\mathcal S_\delta}}f(n)
 =\frac{\gamma}{\rho(\delta^{-1})N^\gamma}
   \sum_{\substack{n\leq N\\n\in\mathcal S_\delta}}
       \frac{f(n)}{n^{1-\gamma}}+o_{Q,c}(1).
\end{align}
The right-hand side is independent of $Q$ up to the error term. Comparing
with $Q=1$ proves multiplicative pseudorandomness, and
Theorem~\ref{thm:main} gives the claimed recurrence.

The same asymptotic also records the local density needed in
Theorem~\ref{thm:suff}. For every fixed $Q\in\N$, we have
\begin{align}\label{eq:pscond}
 |\{n\leq N:n\in\mathcal P_c\cap\mathcal S_\delta,\ Q\mid n\}|
  =\left(\frac1Q+o_{Q,c}(1)\right)
       |\{n\leq N:n\in\mathcal P_c\cap\mathcal S_\delta\}|.
\end{align}
Moreover, since $|\mathcal P_c\cap[1,N]|\sim N^\gamma$, we obtain
\begin{align}\label{eq:pscond2}
 |\{n\leq N:n\in\mathcal P_c\cap\mathcal S_\delta\}|
  =(\rho(\delta^{-1})+o_c(1))|\mathcal P_c\cap[1,N]|.
\end{align}

To prove the proposition, we use a factorization of smooth integers followed
by Heath--Brown's bilinear estimate. We first record the factorization that
separates the smoothness conditions on the two variables.

\begin{lemma}\label{lem:ps-factorization}
Let $U\geq1$. Every integer $n>U$ has a unique factorization $n=d\ell$
satisfying
\begin{align*}
 U<d\leq U P^-(d),\qquad P^+(\ell)\leq P^-(d).
\end{align*}
In this factorization, we have $P^+(n)=P^+(d)$. If $P^+(n)\leq y$, then
$d\leq Uy$.
\end{lemma}

\begin{proof}
Arrange the prime factors of $n$, counted with multiplicity, in
non-increasing order, and let $d$ be their shortest initial product exceeding
$U$. The preceding product is $d/P^-(d)\leq U$, while every remaining prime
factor is at most $P^-(d)$. Conversely, the two displayed conditions force
$d$ to be this initial product, which proves uniqueness and the remaining
assertions.
\end{proof}

The next lemma reduces the Piatetski--Shapiro indicator to exponential sums.
Its proof follows the Fourier approximation used in
\cite[Section~2]{MR698168}.

\begin{lemma}\label{lem:ps-sawtooth}
Let $\frac12<\gamma<1$ and fix $Q\in\N$. Suppose that $X\geq2$ and
$H\geq2$ is an integer. For any complex numbers $u_n$ satisfying
$|u_n|\leq1$, define
\begin{align*}
 S_h(t):=\sum_{X<n\leq t}u_n e\bigl(h(Qn)^\gamma\bigr).
\end{align*}
Then, uniformly for $X\leq T\leq2X$, we have
\begin{gather*}
\sum_{X<n\leq T}u_n
    \bigl(\psi(-(Qn+1)^\gamma)-\psi(-(Qn)^\gamma)\bigr)\\
 \ll_{Q,\gamma}
 X^{\gamma-1}\max_{X\leq t\leq2X}
       \sum_{1\leq|h|\leq H}|S_h(t)|
 +\frac{X\log(2H)}H+X^{1-\gamma}+(HX^\gamma)^{1/2}.
\end{gather*}
\end{lemma}

\begin{proof}
The standard Fourier approximation \cite[p.~245]{MR698168} gives
\begin{align*}
 \psi(t)=-\sum_{0<|h|\leq H}\frac{e(ht)}{2\pi ih}+O(g(t,H)),
\end{align*}
where
\begin{align*}
 g(t,H):=\min\left\{1,\frac1{H\|t\|}\right\}.
\end{align*}
Here $g(t,H)=\sum_{h\in\Z}c_he(ht)$ with
\begin{align*}
 c_0\ll&\frac{\log(2H)}H,\\
 c_h\ll&\min\left\{\frac{\log(2H)}H,\frac1{|h|},
\frac H{|h|^2}\right\}\qquad (h\ne0) .
\end{align*}
For $a\in\{0,1\}$ and $\lambda\neq0$, the derivative estimate in \cite[Lemma~1]{MR698168} gives
\begin{align*}
 \sum_{X<n\leq T}e\bigl(\lambda(Qn+a)^\gamma\bigr)
 \ll_{Q,\gamma}|\lambda|^{-1}X^{1-\gamma}
                     +( |\lambda|X^\gamma)^{1/2}
\end{align*}
Summing against the Fourier
coefficients, we obtain
\begin{align*}
 \sum_{X<n\leq2X}g(-(Qn+a)^\gamma,H)
 \ll_{Q,\gamma}\frac{X\log(2H)}H+X^{1-\gamma}+(HX^\gamma)^{1/2}.
\end{align*}
To treat the finite Fourier sum, put
\begin{align*}
 \Delta_Q(x)&:=(Qx+1)^\gamma-(Qx)^\gamma,\\
 b_h(x)&:=\frac{1-e(-h\Delta_Q(x))}{2\pi ih}.
\end{align*}
For $X\leq x\leq2X$, we have
\begin{align*}
 |b_h(x)|\ll_{Q,\gamma}X^{\gamma-1},\qquad
 |b_h'(x)|\ll_{Q,\gamma}X^{\gamma-2}.
\end{align*}
The finite Fourier contribution is
\begin{align*}
\sum_{0<|h|\leq H}\sum_{X<n\leq T}u_nb_h(n)e(-h(Qn)^\gamma). 
\end{align*}
Partial summation gives the first term in the claimed bound, while the
preceding estimate for $g$ bounds the remainder.
\end{proof}

We shall apply Heath--Brown's Type II estimate with an interval depending
on one factor. For completeness, we include the verification that this additional dependence
does not change his argument.

\begin{lemma}\label{lem:ps-typeii}
Let $\eta>0$ and $\frac12+\eta\leq\gamma<1$. For $Z\geq2$, put
\begin{align*}
 H:=\left\lfloor Z^{1-\gamma}(\log Z)^3\right\rfloor,\qquad
 U:=Z^{1-\gamma+\eta},\qquad V:=Z^{5\gamma-4-\eta}.
\end{align*}
For each positive integer $d$, let $\mathcal I_d\subseteq(Z,2Z]$ be an
interval. Suppose that $|a_d|,|b_s|\leq1$ and that $a_d=0$ unless
$U<d<V$. Then
\begin{align*}
 \sum_{1\leq|h|\leq H}
 \left|\sum_{\substack{d,s\geq1\\ds\in\mathcal I_d}}
      a_db_s e\bigl(h(ds)^\gamma\bigr)\right|
 \ll_{\gamma,\eta}Z^{1-\eta/7}.
\end{align*}
\end{lemma}

\begin{proof}
The negative values of $h$ are handled by conjugating the coefficients,
and hence it suffices to consider positive $h$. Choose $|\varepsilon_h|\leq1$ to remove the absolute
values in the sum over $h$. After interchanging the sums and using
$|b_s|\leq1$, the resulting expression is at most
\begin{align*}
 \sum_s\left|\sum_{\substack{U<d<V\\ds\in\mathcal I_d}}
       a_d\sum_{h\leq H}\varepsilon_h e\bigl(h(ds)^\gamma\bigr)\right|.
\end{align*}
We follow the proof of \cite[Lemma~4]{MR698168}, retaining the dependence
of the interval on $d$. Decompose the variables into dyadic ranges satisfying
$d\asymp R$ together with $s\asymp S$ and $h\asymp J$, where
$RS\asymp Z$ and $J\leq H$. This decomposition incurs only logarithmic
factors.
For one such choice of ranges, put
\begin{align*}
 L:=\sum_{S<s\leq2S}
 \left|\sum_{\substack{d\asymp R,\ h\asymp J\\
                       ds\in\mathcal I_d}}
 a_d\varepsilon_h e\bigl(h(ds)^\gamma\bigr)\right|.
\end{align*}

We first require the spacing bound from \cite[pp.~256--257]{MR698168} that
\begin{gather*}
 \mathcal N(\Delta):=|\{(d,r,h,k):d,r\asymp R,\ h,k\asymp J,
                   \ |hd^\gamma-kr^\gamma|\leq\Delta\}|,\\
 \mathcal N(\Delta)\ll_\gamma \Delta JR^{2-\gamma}+JR\log(2JR).
\end{gather*}
To verify it, put $a:=\gcd(d,r)$ and $b:=\gcd(h,k)$. Fix $a,b$. For each
of the $O(J^2/b^2)$ pairs $(h,k)$ with $\gcd(h,k)=b$, the relevant ratios
$d/r$ lie in an interval of length
$O_\gamma(\Delta/(JR^\gamma))$. Distinct such ratios with $(d,r)=a$ are
separated by at least $a^2/(4R^2)$. Hence the number of quadruples is
\begin{align*}
 \ll_\gamma \frac{J^2}{b^2}
          +\frac{\Delta JR^{2-\gamma}}{a^2b^2}.
\end{align*}
Interchanging the two pairs gives the same bound with $R^2/a^2$ in
place of $J^2/b^2$. Summing the smaller bound over $a$ and $b$, and using
\begin{align*}
 \sum_{a\leq2R}\sum_{b\leq2J}
       \min\left\{\frac{J^2}{b^2},\frac{R^2}{a^2}\right\}
       \ll JR\log(2JR)
\end{align*}
proves the claim.

Let $K\geq1$ be an integer, and partition the pairs $(d,h)$ into
sets $\mathcal P_1,\ldots,\mathcal P_m$, where $m\ll K$, such that
$hd^\gamma$ ranges over an interval of length
$O(JR^\gamma/K)$ on each $\mathcal P_\ell$. For $S<s\leq2S$, put
\begin{align*}
 F_\ell(s):=
 \sum_{\substack{(d,h)\in\mathcal P_\ell\\
                  ds\in\mathcal I_d}}
 a_d\varepsilon_h e\bigl(hd^\gamma s^\gamma\bigr).
\end{align*}
The Cauchy--Schwarz inequality gives
\begin{align*}
 L^2
 &\ll S\sum_{S<s\leq2S}
       \left|\sum_{\ell\leq m}F_\ell(s)\right|^2\\
 &\ll KS\sum_{\ell\leq m}\sum_{S<s\leq2S}|F_\ell(s)|^2.
\end{align*}
Expanding the last square and using $|a_d|,|\varepsilon_h|\leq1$, we
obtain
\begin{align*}
 L^2\ll KS\sum_{\ell\leq m}
 \sum_{\substack{(d,h),(r,k)\in\mathcal P_\ell}}
 \left|\sum_{s\in(S,2S]\cap d^{-1}\mathcal I_d\cap r^{-1}\mathcal I_r}
 e\bigl((hd^\gamma-kr^\gamma)s^\gamma\bigr)\right|.
\end{align*}
Notice that the range of $s$ in the inner sum is an interval. Put
$\lambda:=hd^\gamma-kr^\gamma$. If $\lambda\neq0$, then
\cite[Lemma~1]{MR698168} gives
\begin{align*}
\sum_{s\in(S,2S]\cap d^{-1}\mathcal I_d\cap r^{-1}\mathcal I_r}
 e\bigl(\lambda s^\gamma\bigr)
 \ll_\gamma
 \min\left\{S,\ |\lambda|^{-1}S^{1-\gamma}
 +( |\lambda|S^\gamma)^{1/2}\right\}.
\end{align*}
For $\lambda=0$, the same sum is trivially bounded by $S$.
Apply the spacing bound to $|\lambda|\leq S^{-\gamma}$ and to the
remaining dyadic ranges up to $O(JR^\gamma/K)$. The square-root term is
bounded using the largest possible $|\lambda|$. For every
$\varepsilon>0$, we have
\begin{align*}
 L^2\ll_{\gamma,\varepsilon}Z^\varepsilon KS\bigg\{
 \left(\frac{JZ^\gamma}K\right)^{1/2}
 \left(\frac{J^2R^2}K+JR\right)
 +S^{1-\gamma}\bigl(JR^{2-\gamma}+JRS^\gamma\bigr)\bigg\}.
\end{align*}
Taking
\begin{align*}
 K:=1+\left\lfloor JS^{(\gamma-2)/3}R^{(\gamma+2)/3}\right\rfloor
\end{align*}
and using $S\asymp Z/R$ together with $J\leq Z^{1-\gamma}(\log Z)^3$,
we arrive at
the six-term bound in \cite[p.~257]{MR698168}:
\begin{align*}
 L^2\ll_{\gamma,\varepsilon}Z^\varepsilon
 \bigl(&Z^{5/2-\gamma}+Z^{3-2\gamma}+Z^{3-\gamma}R^{-1}
       +Z^{(10-5\gamma)/3}R^{1/3}\\
       &+Z^{(8-4\gamma)/3}R^{2/3}
       +Z^{(10-8\gamma)/3}R^{4/3}\bigr).
\end{align*}
Since $R\gg U$ and $R\ll V$, each term in parentheses is
$\ll Z^{2-\eta/3}$. Taking $\varepsilon<\eta/21$ and absorbing the
dyadic decomposition into the remaining power saving proves the lemma.
\end{proof}

We now combine the three lemmas. The smoothness parameter is chosen small
enough that the factor obtained from Lemma~\ref{lem:ps-factorization}
remains within the Type II range.

\begin{proof}[Proof of Proposition~\ref{prop:ps}]
Put
\begin{align*}
 \eta:=\frac{6\gamma-5}{4},\qquad \delta:=\frac\eta{25}.
\end{align*}
Then $\frac12+\eta<\gamma$ and
\begin{align*}
 2\eta+\delta<6\gamma-5,\qquad \frac\eta7>2\delta.
\end{align*}
Fix $Q\in\N$ and take $X$ sufficiently large in terms of $Q$ and $c$.
Set $Z:=QX$, and define $H,U,V$ as in Lemma~\ref{lem:ps-typeii}.
In particular, we have $U<X$ and $U(2Z)^\delta<V$.

For $X<n\leq T\leq2X$ with $Qn\in\mathcal S_\delta$, use the
factorization $n=d\ell$ in Lemma~\ref{lem:ps-factorization} and put
$p:=P^-(d)$. Then
\begin{align*}
 p\leq(2Z)^\delta,\qquad U<d\leq Up<V,\qquad P^+(\ell)\leq p.
\end{align*}
For each such prime $p$, define
\begin{align*}
 a_p(d)&:=f(d)1_{\{U<d\leq Up,\ P^-(d)=p\}},\\
 b_p(Q\ell)&:=f(\ell)1_{\{P^+(\ell)\leq p\}}.
\end{align*}
Set $b_p(s):=0$ when $Q\nmid s$.
Also, put
\begin{align*}
 \mathcal I_d:=(Z,QT]\cap
 \left[\max\{P^+(Q),P^+(d)\}^{1/\delta},\infty\right).
\end{align*}
Since $P^+(n)=P^+(d)$, this interval imposes precisely the condition
$Qn\in\mathcal S_\delta$. Uniqueness of the factorization gives
\begin{align*}
 S_h(X,T)&:=\sum_{\substack{X<n\leq T\\Qn\in\mathcal S_\delta}}
                      f(n)e\bigl(h(Qn)^\gamma\bigr)\\
 &=\sum_{p\leq(2Z)^\delta}
   \sum_{\substack{d,s\geq1\\ds\in\mathcal I_d}}
              a_p(d)b_p(s)e\bigl(h(ds)^\gamma\bigr).
\end{align*}
All coefficients are bounded by $1$, and Lemma~\ref{lem:ps-typeii}
therefore shows that, uniformly for $X\leq T\leq2X$, we have
\begin{align}\label{eq:ps-exp}
 \sum_{1\leq|h|\leq H}|S_h(X,T)|
 \ll_{c}Z^{1-\eta/7+\delta}\ll_{Q,c}X^{1-\delta}.
\end{align}

Notice that the exact indicator identity
\begin{align*}
 1_{\{Qn\in\mathcal P_c\}}
 &=\lfloor-(Qn)^\gamma\rfloor-\lfloor-(Qn+1)^\gamma\rfloor\\
 &=\Delta_Q(n)+\psi(-(Qn+1)^\gamma)-\psi(-(Qn)^\gamma)
\end{align*}
also holds when either endpoint is an integer. Apply
Lemma~\ref{lem:ps-sawtooth} with $u_n=f(n)1_{\{Qn\in\mathcal S_\delta\}}$.
Together with \eqref{eq:ps-exp}, it bounds the contribution of the two
sawtooth terms on $(X,T]$ by
\begin{align*}
 O_{Q,c}\left(X^{\gamma-\delta}
       +\frac{X^\gamma}{(\log X)^2}
       +X^{1-\gamma}+X^{1/2}(\log X)^{3/2}\right)
       =o_{Q,c}(X^\gamma).
\end{align*}
Summing over dyadic intervals above $N^{1/2}$ and bounding the remaining
terms trivially, we conclude that
\begin{align*}
 \sum_{\substack{n\leq N\\Qn\in\mathcal P_c\cap\mathcal S_\delta}}f(n)
 =\sum_{\substack{n\leq N\\Qn\in\mathcal S_\delta}}
                         f(n)\Delta_Q(n)+o_{Q,c}(N^\gamma).
\end{align*}
Taylor's theorem gives
\begin{align*}
 \Delta_Q(n)=\gamma Q^{\gamma-1}n^{\gamma-1}
                         +O_{Q,c}(n^{\gamma-2}),
\end{align*}
whose error is summable.

It remains to replace $Qn\in\mathcal S_\delta$ by
$n\in\mathcal S_\delta$. For all sufficiently large $n$, the latter
condition implies the former, since $P^+(Q)$ is fixed. If
$Qn\in\mathcal S_\delta$ but $n\notin\mathcal S_\delta$, then
\begin{align*}
 n^\delta<P^+(n)\leq(Qn)^\delta.
\end{align*}
Chebyshev's bound for primes consequently gives
\begin{align*}
 |\{X<n\leq2X:Qn\in\mathcal S_\delta,\ n\notin\mathcal S_\delta\}|
 \leq\sum_{X^\delta<p\leq(2QX)^\delta}\frac{2X}{p}
 \ll_{Q,c}\frac X{\log X}.
\end{align*}
By dyadic summation, the corresponding weighted difference is
\begin{align*}
 \sum_{n\leq N}n^{\gamma-1}
   \left|1_{\{Qn\in\mathcal S_\delta\}}
                -1_{\{n\in\mathcal S_\delta\}}\right|
 \ll_{Q,c}N^{\gamma/2}+\frac{N^\gamma}{\log N}
 =o_{Q,c}(N^\gamma).
\end{align*}
Combining these estimates proves the proposition, with all errors uniform
for $f\in\M$.
\end{proof}

\section{Proof of Theorem~\ref{thm:nilbohr}}\label{sec:nilbohr}

We estimate the mean of a multiplicative function on a regular irrational
nil-Bohr set by applying the orthogonality theorem of
Frantzikinakis--Host. The uniformity in the multiplicative function will
also be used in Section~\ref{sec:dist}.

\begin{proposition}\label{prop:nilbohr}
Let $R=\mathcal B_{X,g}(A)$ be a regular irrational nil-Bohr set, and put
$\delta=m_X(A)$. For every fixed $Q\in\N$, we have
\begin{align}\label{eq:weighted-main}
 \sum_{\substack{n\leq N\\Qn\in R}}f(n)
 =\delta\sum_{n\leq N}f(n)+o_{Q,X,g,A}(N)
\end{align}
as $N\to\infty$, uniformly for $f\in\M$.
\end{proposition}

We first deduce Theorem~\ref{thm:nilbohr} from the proposition. Put
\begin{align*}
 D_Q(N):=|R_Q\cap[1,N]|.
\end{align*}
Taking $f\equiv1$ in Proposition~\ref{prop:nilbohr} gives
$D_Q(N)=\delta N+o_{Q,X,g,A}(N)$. Since $|f|=1$, division in
\eqref{eq:weighted-main} shows that, uniformly for $f\in\M$, we have
\begin{align}\label{eq:nilbohr-mean}
 \E_{\substack{n\leq N\\Qn\in R}}f(n)
 =\E_{n\leq N}f(n)+o_{Q,X,g,A}(1).
\end{align}
Comparing this identity with the case $Q=1$ proves multiplicative
pseudorandomness, and Theorem~\ref{thm:main} gives recurrence. Moreover,
we have
\begin{align}\label{eq:bscond}
 \frac{|R\cap Q\N\cap[1,N]|}{|R\cap[1,N]|}
 =\frac{D_Q(\lfloor N/Q\rfloor)}{D_1(N)}
 \longrightarrow\frac1Q \qquad \text{as $N \to \infty.$}
\end{align}
The estimate $|R\cap[1,N]|\sim\delta N$ also verifies
\eqref{eq:regular}. Hence the positive relative density conclusion of
Theorem~\ref{thm:suff} applies to these sets.

To prove the proposition, we use the following orthogonality result.

\begin{lemma}[Frantzikinakis--Host]\label{lem:nil-orthogonality}
Let $X=G/\Gamma$ be a nilmanifold, and let $g:\Z\to G$ be a polynomial
sequence whose orbit is totally equidistributed in $X$. If $F:X\to\C$
is bounded and Riemann integrable with $\int_XF\,dm_X=0$, then
\begin{align*}
 \lim_{N\to\infty}
 \sup_{\substack{f \in \M}}
 \left|\E_{n\leq N}f(n)F(g(n)\Gamma)\right|=0.
\end{align*}
\end{lemma}

\begin{proof}
This follows from \cite[Theorem~2.2]{MR3556289} and its extension to
Riemann integrable functions in
\cite[Remark~(3)]{MR3556289}.
\end{proof}

\begin{proof}[Proof of Proposition~\ref{prop:nilbohr}]
Fix $Q\in\N$. The polynomial sequence $g_Q(n):=g(Qn)$ is totally
equidistributed. Indeed, its restriction to the progression $a\N+b$ is
$(g(Qan+Qb)\Gamma)_{n\in\N}$, which is equidistributed by the total
equidistribution of $g$. Moreover, the function $1_A-\delta$ is Riemann integrable
and has integral zero. Lemma~\ref{lem:nil-orthogonality} therefore gives
\begin{align*}
\lim_{N \to \infty} \sup_{f\in\M}
 \left|\E_{n\leq N}f(n)(1_R(Qn)-\delta)\right|=0.
\end{align*}
Multiplying by $N$ proves \eqref{eq:weighted-main}.
\end{proof}

\section{Proof of Corollary~\ref{cor:genpoly}}\label{sec:genpoly}

Corollary~\ref{cor:genpoly} follows by a short extension of the proof of
Theorem~\ref{thm:nilbohr}. The only additional point is that the
Bergelson--Leibman representation may retain a periodic component.
The following lemma treats this component on each residue class.

\begin{lemma}\label{lem:gp-short}
Let $u:\Z\to\R$ be a bounded generalized polynomial, and suppose that
there exists $\delta\in\R$ such that for any $a\in\N, \ b\in\Z$, we have
\begin{align*}
 \lim_{N\to\infty}\E_{n\leq N}u(an+b)=\delta.
\end{align*}
Then
\begin{align*}
 \lim_{N\to\infty}
 \sup_{f\in\M}
 \left|\E_{n\leq N}f(n)(u(n)-\delta)\right|=0.
\end{align*}
\end{lemma}

\begin{proof}
After an affine rescaling, the representation theorem of
Bergelson--Leibman \cite[Theorem~A$^*$]{MR2318563} gives
\begin{align*}
 u(n)=F(a(n)\Gamma),
\end{align*}
where $G/\Gamma$ is a connected nilmanifold. Here $a:\Z\to G$ is
polynomial and $F$ is bounded and piecewise polynomial. By the
factorization theorem \cite[Corollary~1.12]{MR2877065}, there exist a
connected rational subgroup $H\leq G$, a constant $\varepsilon\in G$, and
polynomial sequences $h:\Z\to H$ and $\gamma:\Z\to G$ such that
\begin{align*}
 a(n)=\varepsilon h(n)\gamma(n).
\end{align*}
The orbit $h(n)(H\cap\Gamma)$ is totally equidistributed in
$H/(H\cap\Gamma)$, while $\gamma(n)\Gamma$ is rational and periodic.
Let $q$ be a period and, for $0\leq r<q$, choose a rational element $\gamma_r\in G$
such that
\begin{align*}
 \gamma(n)\Gamma=\gamma_r\Gamma\qquad(n\equiv r\Mod{q}).
\end{align*}

Put
\begin{align*}
 H_r:=\gamma_r^{-1}H\gamma_r,\quad \Gamma_r:=H_r\cap\Gamma,\quad
 h_r(n):=\gamma_r^{-1}h(n)\gamma_r,\quad
 F_r(x\Gamma_r):=F(\varepsilon\gamma_r x\Gamma).
\end{align*}
Then $u(n)=F_r(h_r(n)\Gamma_r)$ whenever $n\equiv r\Mod{q}$.
The orbit $h_r(n)\Gamma_r$ is totally equidistributed in
$X_r:=H_r/\Gamma_r$ by \cite[Corollary~5.5]{MR3556289}. By
\cite[Proposition~2.8]{MR2318563}, the coordinate map defining $F_r$
is piecewise polynomial. Hence $F_r$ is piecewise polynomial and, as
noted in \cite[Section~0.18]{MR2318563}, Riemann integrable. Therefore,
we have
\begin{align*}
 \int_{X_r}F_r\,dm_{X_r}
 =\lim_{N\to\infty}\E_{n\leq N}u(qn+r)=\delta.
\end{align*}
Applied to $F_r-\delta$, the progression estimate
\cite[Theorem~6.1]{MR3556289} gives the required bound for Lipschitz
functions. The $L^1$-approximation argument in \cite[Remark~(3)]{MR3556289} following
\cite[Theorem~2.2]{MR3556289} applies with the progression indicator, since
$|f|\leq1$, and gives
\begin{align*}
 \sup_{f\in\M}
 \left|\frac1N\sum_{\substack{n\leq N\\n\equiv r\Mod{q}}}
 f(n)(F_r(h_r(n)\Gamma_r)-\delta)\right|=o(1).
\end{align*}
Summing over $r$ proves the lemma.
\end{proof}

\begin{proof}[Proof of Corollary~\ref{cor:genpoly}]
First let $I=[c,d)\subseteq[0,1)$. For each fixed $Q\in\N$, the function
\begin{align*}
 u_Q(n):=1_I(\{p(Qn)\})
 =\lfloor\{p(Qn)\}-c\rfloor-\lfloor\{p(Qn)\}-d\rfloor
\end{align*}
is a bounded generalized polynomial with mean $|I|$ on every
arithmetic progression. Lemma~\ref{lem:gp-short} shows that, uniformly
for $f\in\M$, we have
\begin{align*}
 \sum_{\substack{n\leq N\\Qn\in \mathcal{L}_p(I)}}f(n)
 =|I|\sum_{n\leq N}f(n)+o_{Q,p,I}(N).
\end{align*}
Adjusting the inclusion of an endpoint affects only $o_{Q,p,I}(N)$
terms by the assumed equidistribution. Hence the same estimate holds
for every interval in the statement.
Taking $f=1$ gives
\begin{align*}
 |\{n\leq N:Qn\in \mathcal{L}_p(I)\}|=|I|N+o_{Q,p,I}(N).
\end{align*}
Dividing the preceding two estimates and comparing with $Q=1$ proves
multiplicative pseudorandomness, and Theorem~\ref{thm:main} gives
multiplicative recurrence.
\end{proof}

We conclude this section by applying Corollary~\ref{cor:genpoly} to the
examples stated in the introduction.

\begin{proof}[Verification of Example~\ref{ex:beatty}]
Put $\gamma:=\alpha^{-1}$ and $\eta:=(1-\beta)/\alpha$. For all
sufficiently large positive integers $m$, we have
\begin{align*}
 m\in\mathcal B_{\alpha,\beta}
 \quad\Longleftrightarrow\quad
 [\gamma m+\eta-\gamma,\gamma m+\eta)\cap\Z\ne\varnothing
 \quad\Longleftrightarrow\quad
 \{\gamma m+\eta\}\in(0,\gamma].
\end{align*}
Since $\gamma$ is irrational, the linear polynomial
$p(m):=\gamma m+\eta$ is equidistributed modulo one on every arithmetic
progression. Corollary~\ref{cor:genpoly} therefore proves the claim,
since changing finitely many elements preserves multiplicative
pseudorandomness.
\end{proof}

\begin{proof}[Verification of Example~\ref{ex:bracket}]
We use the Heisenberg realization in \cite[Section~3]{younis2026sets}.
Let $G=\R^3$ with multiplication
\begin{align*}
 (x,y,z)(x',y',z')=(x+x',y+y',z+z'+xy'),
\end{align*}
and let $\Gamma=\Z^3$. The sequence $g(n):=(\beta n,\alpha n,0)$ is
polynomial with respect to the lower central series. Its horizontal
projection $(\beta n,\alpha n)\Mod{\Z^2}$ is equidistributed on every
arithmetic progression by the rational independence of $1,\alpha,\beta$.
Leibman's criterion \cite[Theorem~1.9]{MR2877065} therefore shows that
$(g(n)\Gamma)_{n\in\Z}$ is totally equidistributed in $X=G/\Gamma$.

For every non-zero integer $h$, the function
\begin{align*}
 F_h((x,y,z)\Gamma):=e\bigl(h(x\lfloor y\rfloor-z)\bigr)
\end{align*}
is well defined and Riemann integrable on $X$. On the fundamental domain
$[0,1)^3$, it equals $e(-hz)$ and hence has integral zero. Since
$F_h(g(n)\Gamma)=e(h\beta n\lfloor\alpha n\rfloor)$, Weyl's criterion
shows that the generalized polynomial $p(n):=\beta n\lfloor\alpha n\rfloor$
is equidistributed modulo one on every arithmetic progression.
Corollary~\ref{cor:genpoly} now proves the claim.
\end{proof}

\begin{proof}[Verification of Example~\ref{ex:polynomial}]
For every $a\in\N$ and $b\in\Z$, the polynomial $P(ax+b)$ has an
irrational non-constant coefficient. Otherwise, the substitution
$x\mapsto(x-b)/a$ would give $P-P(0)\in\Q[x]$. Weyl's theorem
\cite[Chapter~1, Theorem~3.2]{MR419394} therefore verifies the hypothesis
of Corollary~\ref{cor:genpoly}, which proves the claim.
\end{proof}

\section{Value distribution of unimodular completely multiplicative functions}
\label{sec:dist}

Inspired by the resolution of the K\'atai--Subbarao conjecture by
Klurman and Mangerel \cite{MR3856825}, we determine the limit points of
unimodular completely multiplicative functions along multiplicatively
pseudorandom sets (see also \cite{leung2026valuedistributionmultiplicativefunctions, tardy2026twopointcorrelationsmultiplicativefunctions}). For a sequence of complex numbers $(z_n)$, we denote
its set of limit points by
\begin{align*}
 \omega(z_n):=\bigcap_{N=1}^{\infty}
 \overline{\{z_n:n\geq N\}}.
\end{align*}
Whenever a sequence is indexed by an infinite subset of $\N$, its terms
are taken in increasing order. For $k\in\N$, we also write
\begin{align*}
 \mu_k:=\{z\in\T:z^k=1\}.
\end{align*}
The \textit{order} of $f\in\M$ is the least integer $k \in \N$ for which
$f^k\equiv1$, if such an integer exists.

\begin{theorem}\label{thm:densepseudo}
Let $R\subseteq\N$ be multiplicatively pseudorandom and let $f\in\M$.
Then
\begin{align*}
\omega\bigl((f(n))_{n\in R}\bigr)=
\begin{cases}
\mu_k&\mbox{{\normalfont if $f$ has finite order $k$,}}\\
\T&\mbox{{\normalfont if $f$ has infinite order.}}
\end{cases}
\end{align*}
\end{theorem}

In particular, the theorem recovers the case $r=1$ of
Corollary~\ref{cor:heil}. Applying it to the completely multiplicative
function $f(n)=e(\Omega(n)/q)$ gives the following consequence.

\begin{corollary}
Let $R\subseteq\N$ be multiplicatively pseudorandom. For every positive
integer $q$ and every residue class $a\Mod{q}$, there exist infinitely many
$n\in R$ satisfying
\begin{align*}
\Omega(n)\equiv a\Mod{q},
\end{align*}
where $\Omega(n)$ counts prime factors with multiplicity.
\end{corollary}

Theorem~\ref{thm:ps} supplies a multiplicatively pseudorandom subset of
the Piatetski--Shapiro set, which is sufficient for the same conclusion
about limit points.

\begin{corollary}
Let $1<c<6/5$ and let $f\in\M$. Then
\begin{align*}
\omega\bigl((f(n))_{n\in\mathcal P_c}\bigr)=
\begin{cases}
\mu_k&\mbox{{\normalfont if $f$ has finite order $k$,}}\\
\T&\mbox{{\normalfont if $f$ has infinite order.}}
\end{cases}
\end{align*}
\end{corollary}

Theorem~\ref{thm:nilbohr} gives the analogous conclusion for regular irrational nil-Bohr sets.

\begin{corollary}
Let $\mathcal{B}$ be a regular irrational nil-Bohr set and let $f\in\M$. Then
\begin{align*}
\omega\bigl((f(n))_{n\in \mathcal{B}}\bigr)=
\begin{cases}
\mu_k&\mbox{{\normalfont if $f$ has finite order $k$,}}\\
\T&\mbox{{\normalfont if $f$ has infinite order.}}
\end{cases}
\end{align*}
\end{corollary}

To examine equidistribution, we recall the \textit{pretentious distance} from
the theory of multiplicative functions \cite{gs}.

\begin{definition}[Pretentious distance]
For $f,g\in\M$, define
\begin{align*}
\mathbb D(f,g)^2:=\sum_p\frac{1-\Re(f(p)\overline{g(p)})}{p}
=\frac12\sum_p\frac{|f(p)-g(p)|^2}{p}.
\end{align*}
We say that $f$ \textit{pretends} to be $g$, and write $f\sim g$, if
$\mathbb D(f,g)<\infty$. Otherwise, we write $f\not\sim g$.\footnote{The ordered pair $(\mathcal{M},\mathbb{D})$ forms an extended metric space, and $\sim$ is an equivalence relation.}
\end{definition}

\begin{definition}[Strictly non-pretentious function]
A function $f\in\M$ is \textit{strictly non-pretentious} if
$f\not\sim g\cdot n^{it}$ for every real $t$ and every finite-order
function $g\in\M$.
\end{definition}

The following criterion is due to Granville and Soundararajan and is
recorded in their unpublished book manuscript
\cite[Theorem~2.1.14]{gs}. We include the proof for completeness,
together with the equivalent formulation in terms of powers.

\begin{proposition}[Granville--Soundararajan]\label{prop:equid}
Let $f\in\M$. The sequence $(f(n))_{n\geq1}$ is equidistributed on
$\T$ if and only if $f$ is strictly non-pretentious. Equivalently, we have
\begin{align}\label{eq:powerpret}
f^j\not\sim n^{it}\qquad\text{for every $j \in \N$ and every $t\in\R$.}
\end{align}
\end{proposition}

\begin{proof}
By Weyl's criterion, Hal\'asz's theorem
\cite[Theorem~2.1.10]{gs}, and its formulation in
\cite[Exercise~2.1.7]{gs}, equidistribution is equivalent to
\eqref{eq:powerpret}. The exceptional case at the prime $2$ in the
general form of Hal\'asz's theorem would require
$f(2)^j=-2^{it}$ and $f(4)^j=-4^{it}$, contrary to complete
multiplicativity.

If $g^j\equiv1$, then
\begin{align*}
\mathbb D(f^j,n^{ijt})\leq j \,\mathbb D(f,g\cdot n^{it}).
\end{align*}
Conversely, suppose that $f^j\sim n^{it}$. At each prime $p$, choose
$g(p)$ to be a nearest $j$-th root of unity to $f(p)p^{-it/j}$, and
extend $g$ completely multiplicatively. Write
$f(p)\overline{g(p)}p^{-it/j}=e^{i\theta_p}$, where
$|\theta_p|\leq\pi/j$. Then
\begin{align*}
|1-f(p)\overline{g(p)}p^{-it/j}|
=|1-e^{i\theta_p}|\leq|1-e^{ij\theta_p}|
=|1-f(p)^jp^{-it}|.
\end{align*}
Consequently, we have $\mathbb D(f,g\cdot n^{it/j})<\infty$, which
proves the claimed equivalence.
\end{proof}

The same criterion holds along the smooth Piatetski--Shapiro sets considered above.

\begin{theorem}\label{thm:psheil}
Let $1<c<6/5$, and choose $\delta>0$ as in Theorem~\ref{thm:ps}. For every
$f\in\M$, the sequence $(f(n))_{n\in\mathcal P_c\cap\mathcal S_\delta}$
is equidistributed on $\T$ if and only if $f$ is strictly
non-pretentious.\footnote{Adapting the argument in
\cite[Section~2]{MR3705270}, one can show that without the smoothness
restriction, Theorem~\ref{thm:psheil} is false in general.}
\end{theorem}

The nil-Bohr mean-value estimate gives the analogous criterion for
regular irrational nil-Bohr sets.

\begin{theorem}\label{thm:nilbohrheil}
Let $\mathcal{B}$ be a regular irrational nil-Bohr set. For every $f\in\M$, the
sequence $(f(n))_{n\in \mathcal{B}}$ is equidistributed on $\T$ if and only if
$f$ is strictly non-pretentious.
\end{theorem}

We first prove the assertion about limit points by combining the
$Q$-trick with the Erd\H{o}s--Tur\'an inequality as in \cite[Sections 5--6]{leung2026valuedistributionmultiplicativefunctions}.

\begin{proof}[Proof of Theorem~\ref{thm:densepseudo}]
For every sufficiently large $K$ and all sufficiently large $N$, define
the probability measure
\begin{align*}
 \nu_{K,N}:=\E_{Q\in\Phi_K}
       \E_{\substack{n\leq N\\Qn\in R}}\delta_{f(Qn)}
\end{align*}
on $\T$. For every fixed positive integer $j$, multiplicative
pseudorandomness gives
\begin{align*}
 \widehat\nu_{K,N}(j)
 &=\E_{Q\in\Phi_K}f(Q)^j
       \E_{\substack{n\leq N\\Qn\in R}}f(n)^j\\
 &=\E_{Q\in\Phi_K}f(Q)^j
       \E_{\substack{n\leq N\\n\in R}}f(n)^j+o_{K,f,j}(1)
\end{align*}
as $N\to\infty$. Hence, if $f^j\not\equiv1$, Lemma~\ref{lem:mult}
gives
\begin{align}\label{eq:fourier-flat}
 \lim_{K\to\infty}\limsup_{N\to\infty}
 |\widehat\nu_{K,N}(j)|=0.
\end{align}
For later use, observe that for every fixed $K$ and $x\geq1$, we have
\begin{align}\label{eq:tail-mass}
 \lim_{N\to\infty}\E_{Q\in\Phi_K}
 \frac{|\{n\leq N:Qn\in R,\ Qn<x\}|}
 {|R_Q\cap[1,N]|}=0.
\end{align}
Indeed, the numerator is bounded, while each denominator tends to
infinity.

Suppose first that $f$ has infinite order. For every $J\in\N$, the
Erd\H{o}s--Tur\'an inequality
\cite[Chapter~2, Theorem~2.5]{MR419394} gives
\begin{align*}
 \sup_{I\in\mathcal I}|\nu_{K,N}(I)-m_{\T}(I)|
 \ll\frac1J+\sum_{j=1}^J\frac{|\widehat\nu_{K,N}(j)|}{j},
\end{align*}
where $\mathcal I$ is the collection of half-open arcs in $\T$ and
$m_{\T}$ denotes the Haar probability measure on $\T$. Since $f^j\not\equiv1$ for
all $j\geq1$, the expression \eqref{eq:fourier-flat} implies that
\begin{align*}
 \lim_{K\to\infty}\limsup_{N\to\infty}
 \sup_{I\in\mathcal I}|\nu_{K,N}(I)-m_{\T}(I)|=0.
\end{align*}
Together with \eqref{eq:tail-mass}, this shows that every arc of positive
length contains $f(m)$ for arbitrarily large $m\in R$. Therefore, we have
\begin{align*}
 \omega\bigl((f(n))_{n\in R}\bigr)=\T.
\end{align*}

Suppose now that $f$ has finite order $k$. The measure $\nu_{K,N}$ is
supported on $\mu_k$. For $\zeta\in\mu_k$, Fourier inversion gives
\begin{align*}
 \nu_{K,N}(\{\zeta\})
 =\frac1k\sum_{j=0}^{k-1}\zeta^{-j}\widehat\nu_{K,N}(j).
\end{align*}
By the minimality of $k$, we have $f^j\not\equiv1$ for
$1\leq j<k$. It follows from \eqref{eq:fourier-flat} that
\begin{align*}
 \lim_{K\to\infty}\limsup_{N\to\infty}
 \max_{\zeta\in\mu_k}
 \left|\nu_{K,N}(\{\zeta\})-\frac1k\right|=0.
\end{align*}
The expression \eqref{eq:tail-mass} now shows that every element of
$\mu_k$ occurs as $f(m)$ for arbitrarily large $m\in R$. Since all
values of $f$ belong to $\mu_k$, the result follows.
\end{proof}

For Theorem~\ref{thm:psheil}, we require the following weighted version
of Hal\'asz's theorem for $\delta$-smooth integers, which is of independent
interest.

\begin{proposition}\label{prop:weightedhalasz}
Let $0\leq\sigma<1$ and $0<\delta\leq1$, and let $f\in\M$. Then
\begin{align*}
\lim_{N\to\infty}\frac1{N^{1-\sigma}}
\sum_{\substack{n\leq N\\n\in\mathcal S_\delta}}
\frac{f(n)}{n^\sigma}=0
\end{align*}
if and only if $f\not\sim n^{it}$ for every $t\in\R$.
\end{proposition}

\begin{proof}
Put $\rho_\delta:=\rho(1/\delta)$. The argument for
\eqref{eq:ps-smooth-density}, with the smoothness threshold
$(dm)^\delta$, also gives
\begin{align}\label{eq:smoothfixedmult}
|\{m\leq x:dm\in\mathcal S_\delta\}|\sim\rho_\delta x \qquad \text{as $x \to \infty$}
\end{align}
for every fixed $d\in\N$.

Suppose first that $f\not\sim n^{it}$ for every $t\in\R$, and put
$f_N(n):=f(n)1_{P^+(n)\leq N^\delta}$. Notice that the sets
$\mathcal S_\delta\cap[1,N]$ and $\{n\leq N:P^+(n)\leq N^\delta\}$
differ by $o_\delta(N)$ integers. By assumption and compactness, for any fixed $T\geq1$, we have
\begin{align*}
M_N(T):=\min_{|t|\leq2T}\sum_{p\leq N^\delta}
\frac{1-\Re(f(p)p^{-it})}{p}\longrightarrow\infty \qquad \text{as $N \to \infty$}.
\end{align*}
The quantitative Hal\'asz theorem
\cite[Corollary~1]{MR2016245}, applied to $f_N$, therefore gives
\begin{align*}
\frac1N\sum_{\substack{n\leq N\\n\in\mathcal S_\delta}}f(n)
\ll(1+M_N(T))e^{-M_N(T)}+\frac1T+o_\delta(1).
\end{align*}
Letting $N\to\infty$ and then $T\to\infty$ proves the assertion for
$\sigma=0$, and partial summation gives the remaining cases.

Conversely, suppose that $g(n):=f(n)n^{-it}$ satisfies $g\sim1$.
Write $g(p)=e^{i\theta_p}$ with $\theta_p \in (-\pi,\pi]$, and put
\begin{align*}
\varepsilon_K^2:=\sum_{p>K}\frac{\theta_p^2}{p}.
\end{align*}
Since $g\sim1$, we have $\varepsilon_K\to0$ as $K\to\infty$.
Let $g_K\in\M$ agree with $g$ at primes at most $K$ and equal $1$
at larger primes. The Tur\'an--Kubilius inequality
\cite[Proposition~2.2.3]{gs}, applied to
$\sum_{p>K}\theta_pv_p(n)$, gives a phase $u_{K,N}\in\T$ such that
\begin{align*}
\frac1N\sum_{n\leq N}|g(n)-u_{K,N}g_K(n)|\ll\varepsilon_K.
\end{align*}
Splitting at $\varepsilon_KN$ when $\varepsilon_K \in (0,1)$ gives
\begin{align}\label{eq:weighted-truncation}
\frac1{N^{1-\sigma}}\sum_{n\leq N}
\frac{|g(n)-u_{K,N}g_K(n)|}{n^\sigma}
\ll_\sigma\varepsilon_K^{1-\sigma}.
\end{align}
The same bound is immediate when $\varepsilon_K=0$, with $u_{K,N}=1$.

Write $g_K=1*b_K$. Since $b_K(p^j)=g(p)^{j-1}(g(p)-1)$ for
$p\leq K$ and vanishes for $p>K$, we have
\begin{gather*}
\sum_{d\geq1}\frac{|b_K(d)|}{d}<\infty,\\
C_K:=\sum_{d\geq1}\frac{b_K(d)}d
=\prod_{p\leq K}\frac{1-1/p}{1-g(p)/p}.
\end{gather*}
The expression \eqref{eq:smoothfixedmult}, partial summation and
dominated convergence now yield
\begin{align*}
\frac1{N^{1-\sigma}}
\sum_{\substack{n\leq N\\n\in\mathcal S_\delta}}
\frac{g_K(n)}{n^{\sigma-it}}
=\frac{ C_K \rho_\delta}{1-\sigma+it} \, N^{it}+o_{K,\delta,\sigma,t}(1).
\end{align*}
Moreover, the identity
\begin{align*}
\log|C_K|
=\sum_{p\leq K}\sum_{j\geq1}
\frac{\Re(g(p)^j)-1}{jp^j}
\end{align*}
and the absolute convergence of the series on the right show that
$|C_K|$ converges to a positive limit $C(g)$. Combining
these estimates with \eqref{eq:weighted-truncation} and then letting
$K\to\infty$ proves that
\begin{align*}
\lim_{N\to\infty}\frac1{N^{1-\sigma}}
\left|\sum_{\substack{n\leq N\\n\in\mathcal S_\delta}}
\frac{f(n)}{n^\sigma}\right|
=\frac{ C(g) \rho_\delta }{|1-\sigma+it|}>0,
\end{align*}
which completes the proof.
\end{proof}

We can now transfer the ordinary equidistribution criterion to the two families.

\begin{proof}[Proof of Theorem~\ref{thm:psheil}]
Write $\gamma:=1/c$. For every positive integer $j$, the expression
\eqref{eq:sumps} with $Q=1$ gives
\begin{align*}
\E_{\substack{n\leq N\\n\in\mathcal P_c\cap\mathcal S_\delta}}f(n)^j
=\frac{\gamma}{\rho(1/\delta)N^\gamma}
\sum_{\substack{n\leq N\\n\in\mathcal S_\delta}}
\frac{f(n)^j}{n^{1-\gamma}}+o(1).
\end{align*}
By Proposition~\ref{prop:weightedhalasz}, these means tend to zero
for every $j\geq1$ if and only if \eqref{eq:powerpret} holds. The
negative Fourier coefficients are their conjugates, and Weyl's criterion
completes the proof.
\end{proof}

\begin{proof}[Proof of Theorem~\ref{thm:nilbohrheil}]
For every positive integer $j$, the expression \eqref{eq:nilbohr-mean}
with $Q=1$ gives
\begin{align*}
\E_{\substack{n\leq N\\n\in \mathcal{B}}}f(n)^j
=\frac1N\sum_{n\leq N}f(n)^j+o(1).
\end{align*}
The negative Fourier coefficients are the conjugates of the positive
ones. Weyl's criterion and Proposition~\ref{prop:equid} therefore give
the result.
\end{proof}

\section{Limitations and future directions}

Multiplicative pseudorandomness is sufficient for recurrence, but it is not
necessary. In fact, it excludes every arithmetic progression $a\N+b$ with
$a\geq2$. If $a\nmid b$, then every $Q\in\Phi_K$ is divisible by $a$ once
$K$ is sufficiently large, and hence $Qn\notin a\N+b$ for every $n\in\N$.
Thus the infinitude condition in Definition~\ref{def:mpr} fails.

Suppose that $a\mid b$. Then $a\N+b$ differs from $a\N$ by a finite set.
Choose a prime $p\mid a$ and write $v:=v_p(a)$. Let $z\in\T$ satisfy
$z^v\ne1$, and define $f(n):=z^{v_p(n)}$. The elementary density formula
for $p$-adic valuations gives
\begin{align*}
\lim_{N\to\infty}\frac1N\sum_{n\leq N}f(n)
=&\left(1-\frac1p\right)\sum_{j\geq0}\frac{z^j}{p^j} \\
=&\frac{1-p^{-1}}{1-zp^{-1}}=:M_p(z)\ne0.
\end{align*}
For all sufficiently large $K$ and $Q\in\Phi_K$, the two means in
Definition~\ref{def:mpr} tend to $M_p(z)$ and $z^vM_p(z)$, respectively.
They are different, and hence the mean-value condition fails.

By contrast, Donoso, Le, Moreira, and Sun
\cite[Theorem~1.10]{MR4594405} proved that $a\N+b$ is multiplicatively
recurrent exactly when $a\mid b(b-1)$. Thus Theorem~\ref{thm:main} does not
capture all local congruence phenomena. A natural direction, suggested by
the \textit{$W$-trick} of Green and Tao \cite{MR2415379}, is to formulate a
version that retains the valuations at primes dividing $W$ and the
finite-order multiplicative information away from $W$. Obtaining a positive spectral lower bound
while preserving this local information is the main difficulty. We leave
such a criterion for future work.

\appendix
\section{Multiplicative ergodicity}

For one measure-preserving transformation, the usual definition of
ergodicity in \cite[Chapter~2]{MR2723325} concerns its additive iterates.
Here the parameter law is multiplication: an action is a family
$T=(T_n)_{n\in\N}$ satisfying $T_1=\rm id$  and
$T_{mn}=T_m\circ T_n$ for $m,n \in \N.$

\begin{definition}[Multiplicative ergodicity]
Let $T=(T_n)_{n\in\N}$ be a measure-preserving multiplicative action on
$(X,\mathcal B,\mu)$. Its \textit{invariant $\sigma$-algebra} is
\begin{align*}
 \mathcal I_T:=\{B\in\mathcal B:
 \mu(T_n^{-1}B\mathbin\triangle B)=0
 \text{ for every }n\in\N\}.
\end{align*}
A function $H\in L^2(\mu)$ is \textit{invariant} if $H\circ T_n=H$ in
$L^2(\mu)$ for every $n\in\N$. The action is \textit{multiplicatively ergodic} if every
set in $\mathcal I_T$ has measure zero or one, or equivalently, every invariant function in
$L^2(\mu)$ is constant almost everywhere.
\end{definition}

An invertible multiplicative action extends to $(\Q_{>0},\times)$ by
setting $T_{r/s}:=T_r\circ T_s^{-1}$ for
$r,s\in\N$. The mean ergodic theorem along the multiplicative F\o{}lner
sequence $(\Phi_K)$ gives the first convergence below
\cite[Proposition~4.8]{frantzikinakis2025decompositionresultsmultiplicativeactions}:
\begin{gather*}
 \lim_{K\to\infty}\E_{Q\in\Phi_K}H\circ T_Q
 =\E(H\mid\mathcal I_T)
 \qquad\text{in }L^2(\mu),\\
 \lim_{K\to\infty}\E_{Q\in\Phi_K}H\circ T_Q=\int_XH\,d\mu
 \qquad\text{in }L^2(\mu),\\
 \lim_{K\to\infty}\E_{Q\in\Phi_K} 
 \int_X \overline G \, (H\circ T_Q)\,d\mu
 =\overline{\left(\int_XG\,d\mu\right)}
 \left(\int_XH\,d\mu\right).
\end{gather*}
The action is ergodic if and only if the second convergence holds for
every $H\in L^2(\mu)$, and equivalently if and only if the final identity
holds for all $H,G\in L^2(\mu)$. For the converse in the last
characterization, take $G=H$ invariant. The identity gives
$\|H\|_2^2=|\int_XH\,d\mu|^2$, which forces $H$ to be constant
almost everywhere.

For $f\in\M$, let $T_n^f(z):=f(n)z$ on $(\T,m_{\T})$, and put
$X_f:=\overline{f(\N)}$. The maps $(T_n^f)_{n\in\N}$ form a
measure-preserving multiplicative action on the circle.

Classical Fourier analysis gives the following characterization, which connects
multiplicative ergodicity with the order appearing in
Section~\ref{sec:dist}.

\begin{theorem} \label{thm:mult-ergodic}
Let $f\in\M$. The multiplicative action $(T_n^f)_{n\in\N}$ on
$(\T,m_{\T})$ is ergodic if and only if $X_f=\T$, or equivalently if and only if
$f^k\not\equiv1$ for every $k \in \N$.
\end{theorem}

\begin{proof}
The closure of a sub-semigroup of a compact group is a subgroup. Since the
closed subgroups of $\T$ are $\T$ and the finite cyclic groups, we have
$X_f=\T$ precisely when $f$ has infinite order. If $f^k\equiv1$, then
the non-constant function $H(z):=z^k$ is invariant under every $T_n^f$.
Conversely, suppose that $f$ has infinite order and that
$H\in L^2(m_{\T})$ is invariant. For every $j\in\Z$ and $n\in\N$, we have
\begin{align*}
 \widehat H(j)=f(n)^j\widehat H(j) .
\end{align*}
For every  $j \neq 0$, there exists $n\in\N$ satisfying
$f(n)^j\ne1$. It follows that $\widehat H(j)=0$ for every $j \neq 0$.
Thus $H$ is constant almost everywhere, and the action is ergodic.
\end{proof}

\section*{Acknowledgements}
The author would like to thank Andrew Granville, Oleksiy Klurman, James
Maynard, Andrei Shubin, Yu-Chen Sun,
Christian T\'afula, and Fei Xue for many helpful discussions, and acknowledges the use of ChatGPT for performing sanity checks and for rephrasing. The author
is supported by a Croucher Fellowship for Postdoctoral Research.

%\nocite{*}
\begingroup
\emergencystretch=4em
\printbibliography
\endgroup

%\vspace{-\baselineskip}

\end{document}